\documentclass[final,10pt]{amsart}

\usepackage{mypreamble}
\usepackage[notref,notcite]{showkeys}

\newtheorem{thm}[subsection]{Theorem}
\newtheorem{lem}[subsection]{Lemma}
\newtheorem{prop}[subsection]{Proposition}
\newtheorem{cor}[subsection]{Corollary}

\theoremstyle{remark}
\newtheorem{rem}[subsection]{Remark}

\theoremstyle{definition}

\numberwithin{equation}{subsection}

\nc\mon{{\on{mon}}}
\nop\Four{Four}

\DeclareMathOperator*{\prolim}{``{\ilim}''}

\title{A new Fourier transform}
\author{Jonathan Wang}

\begin{document}
\maketitle

\section{Introduction}
In order to define a geometric Fourier transform, one usually works with either
$\ell$-adic sheaves in characteristic $p>0$ or with $\eD$-modules in characteristic $0$
(under these conditions one has a rank $1$ local system on $\bbA^1$ which plays the
role of the function $e^{ix}$ in classical Fourier analysis). 
If one only needs to consider homogeneous sheaves, however, Laumon \cite{Laumon} 
provides a uniform geometric construction of the Fourier transform for $\ell$-adic sheaves
in any characteristic. 
Laumon considers homogeneous sheaves as sheaves on the stack quotient of a vector bundle $V$
by the homothety $\bbG_m$ action. This category is closely related to the
category of (unipotently) monodromic sheaves on $V$ (cf.~\cite{BY}). 
While it has been well known to experts that a similar uniform construction 
of the Fourier transform exists for monodromic sheaves (Beilinson suggests
a definition in \cite[footnote 2]{Beilinson2}), the details have not been exposited 
in the literature. In this note, we fill in this gap. 
We also introduce a new functor, which is defined on all sheaves in any characteristic, and
show that it agrees with the usual Fourier transform on monodromic sheaves.

We define the new Fourier transform $\Four_B$ in \S \ref{sect:F^2} and show that
the ``square'' $\Four_B^2$ has a simple formula. In \S \ref{sect:main}, 
we use this formula to prove the main result that
$\Four_B$ induces an equivalence of bounded derived categories of monodromic (\'etale) sheaves. 
We also discuss the relation between $\Four_B$ and Laumon's homogeneous Fourier transform.
In \S \ref{sect:FD}, we compare $\Four_B$ and the Fourier-Deligne
transform in characteristic $p>0$. 
Our study of $\Four_B$ reveals several surprising facts about a certain object 
$j^*B$ of the monoidal category $D_{ctf}(\bbG_m)$. 
In \S \ref{sect:Dmod}, we prove the analogous facts about $j^*B$ in the $\eD$-module setting by
considering the Mellin transform. We use this to show that $\Four_B$
agrees with the Fourier transform on monodromic $\eD$-modules.

\subsection{Acknowledgements} The research was partially supported
by the Department of Defense (DoD) through the NDSEG fellowship.
The author is very thankful to Sasha Beilinson and 
Vladimir Drinfeld for many helpful discussions. The definition of $\Four_B$
was first suggested by Drinfeld.

\subsection{Notation and terminology}
Let $k$ be an arbitrary base field and fix an algebraic closure $\bar k$. 
Choose a prime $\ell$ not equal to the characteristic
of $k$. Let $R$ be a finite commutative $\bbZ/\ell^r$-algebra for a positive integer $r$. 
Fix a base scheme $S$ of finite type over $k$. Let $\pi:V \to S$ be a vector bundle of rank $d$ 
and $\pi^\vee : V^\vee \to S$ the dual vector bundle. 
We will work with the bounded derived category 
$D^b_c(V)=D^b_c(V,R)$ of \'etale sheaves of $R$-modules with constructible cohomologies. 
Our results are also true when $D^b_c(V,R)$ is replaced by 
$D_{ctf}(V,R)$ or $D^b_c(V,\wbar\bbQ_\ell)$.
All functors will be assumed to be derived.

We say a complex $M \in D^b_c(V)$ is monodromic if $M$ is monodromic 
in the sense of Verdier \cite{Verdier} after base change to $\bar k$. 
This is equivalent to 
the existence of an integer $n$ coprime to $p$ and an isomorphism 
$\theta(n)^*M \cong \pr_2^*M$ where $\theta(n) : \bbG_m \xt V \to V$  sends
$(\lambda,v)$ to $\lambda^n v$, and $\pr_2 : \bbG_m \xt V \to V$ is the projection 
\cite[Proposition 5.1]{Verdier}.
We denote the monodromic subcategory by $D^b_\mon(V)$. 
We recall the fact that $\pi_! \cong 0^!$ on monodromic complexes (cf.~\cite[Lemme 6.1]{Verdier}
or \cite[Proposition 1]{Springer} for two different methods of proof).

The category $D_{ctf}(\bbG_m)$ is monoidal with respect to convolution 
with compact support, which is defined by 
\[ L * K = m_!(L\bt K) \]
where $m: \bbG_m \xt \bbG_m \to \bbG_m$ is multiplication, and $L,K \in D_{ctf}(\bbG_m)$. 
This monoidal category acts on $D^b_c(V)$ by 
\[ L * M = \theta(1)_!( L \bt M ) \]
where $\theta(1) : \bbG_m \xt V \to V$ is the action map, 
$L \in D_{ctf}(\bbG_m)$, and $M \in D^b_c(V)$.

\section{The functor $\Four_B$ and its square} \label{sect:F^2}

Let $u:\bbA^1 - \{1\} \into \bbA^1$ be the open embedding removing
$1 \in \bbA^1(k)$, and let $j:\bbA^1 - \{0\} \into \bbA^1$ be 
the open embedding removing zero. Define 
\[ B = u_*R \in D_{ctf}(\bbA^1). \] 
One observes that $h_! B = 0$ where $h:\bbA^1 \to \spec k$ is the structure map,
and $0^*B \cong R$ where $0:\spec k \into \bbA^1$.

Define $\Four_{V/S,B} : D^b_c(V) \to D^b_c(V^\vee)$ by
\[ \Four_{V/S,B}(M) = \pr^\vee_! (\pr^*M \ot \mu^* B)[d] \]
where $\pr^\vee : V^\vee \xt_S V \to V^\vee$ and $\pr : V^\vee \xt_S V \to V$ are
the projections and $\mu : V^\vee \xt_S V \to \bbA^1$
is the natural pairing $(\xi,v) \mapsto \brac{v,\xi}$.
This is the new Fourier transform that we will consider. 
Our goal in this section is to prove the following theorem.

\begin{thm} \label{keythm} 
There is a canonical isomorphism 
\[ \Four_{V^\vee/S,B} \circ \Four_{V/S,B}(M) \cong j^*B * M(-d)[1]. \] 
for $M \in D^b_c(V)$.
\end{thm}

Let $\pr',\pr'':V\xt_S V \to V$ be the first and second projections, respectively, and
$\pr_{ij}$ the projection from $V \xt_S V^\vee \xt_S V$ to the product of 
the $i$'th and $j$'th factor.
The usual formal argument shows that $\Four_{V^\vee/S,B} \circ \Four_{V/S,B}$
is isomorphic to
the functor $M \mapsto \pr'_!(\pr''^*M \ot K)$ where 
\[ K = \pr_{13!}(\pr_{12}^* \mu^* B \ot \pr_{23}^*\mu^*B)[2d]. \]
We claim there exists a canonical isomorphism  
\begin{equation}\label{eq:claim}
	K \cong \rho_! \pr_1^*j^* B (-d)[1] 
\end{equation}
where $\rho : \bbG_m \xt V \to V \xt_S V$
is defined by $(\lambda,v) \mapsto (\lambda v, v)$,
and $\pr_1 : \bbG_m \xt V \to \bbG_m$ is the natural projection.
This claim implies the theorem since 
$\pr'_!(\pr''^* M \ot \rho_! \pr_1^*j^*B) \cong j^*B * M$ by the projection formula.

\smallskip

We first establish two lemmas which will help us prove the claim.
\begin{lem}\label{lem1}
 If $v,w \in V(\bar k)$ are not in the same $\bbG_m$-orbit, then
$K_{(v,w)}=0$. 
\end{lem}
\begin{proof}
We can assume $S = \spec \bar k$. 
Clearly $v$ and $w$ cannot both be zero; we will assume
$v \ne 0$. Since $v,w$ are not in the same $\bbG_m$-orbit, there exists
$\xi \in V^\vee(\bar k)$ such that $\brac{w,\xi}=0$ and $\brac{v,\xi}\ne 0$. 
Let $\brac{v} : V^\vee \to \bbA^1_{\bar k}$ denote the evaluation by $v$ map. 
Split $V^\vee$ as $\bar k \xi \oplus H_v$ where $H_v = (\bar kv)^\perp$. 
With respect to this decomposition, $\brac{v}^*B \ot \brac{w}^*B \cong 
B \bt (\brac{w}|_{H_v})^*B$. Then by Kunneth formula,
\[ \pi^\vee_!(\brac{v}^* B \ot \brac{w}^*B) \cong 
h_!B \ot (\pi^\vee|_{H_v})_! (\brac{w}|_{H_v})^*B = 0 \]
Therefore $K_{(v,w)}=0$.
\end{proof}

\begin{lem} \label{lem2}
There is a canonical isomorphism 
\[ J^* K \cong J^*\rho_! \pr_1^* j^*B (-d)[1] \]
where $J:V\xt_S V - 0(S) \into V\xt_S V$ is the open embedding removing zero.
\end{lem}

\begin{proof} We use $V^\circ$ to denote $V - 0(S)$.
In this proof we will use $\rho$ to denote the restricted morphism
$\bbG_m \xt V^\circ \into V \xt_S V$, which is an immersion, and
$\pr_1 : \bbG_m \xt V^\circ \to \bbG_m$ to denote the projection.
From Lemma \ref{lem1} we know that $J^*K$ is supported on the image of $\rho$. 
Thus it suffices to consider $\rho^* J^* K$. 
Define \[ \omega : \bbG_m \xt V^\vee \xt_S V^\circ \to \bbG_m \xt \bbA^1 \xt 
V^\circ \] by sending $(\lambda,\xi,v)$ to $(\lambda, \brac{v,\xi}, v)$.
Then 
\[ \rho^*J^*K \cong \pr_{13!}\omega_!\omega^* \pr_{12}^* (m^* B \ot p_2^*B)[2d] \]
where $\pr_{13},\pr_{12}$ are projections from $\bbG_m \xt \bbA^1 \xt V^\circ$
and $m,p_2 : \bbG_m \xt \bbA^1 \to \bbA^1$ are the multiplication and projection maps.
Since $\omega$ is in fact a vector bundle of rank $d-1$, 
we see that $\omega_!R$ is isomorphic to $R(1-d)[2-2d]$. 
Therefore the projection formula implies that
\[ \rho^*J^*K \cong \pr_{13!} \pr_{12}^*(m^* B \ot p_2^*B)(1-d)[2].  \]
We have a Cartesian square
\[ \xymatrix{ \bbG_m \xt \bbA^1 \xt V^\circ \ar[r]^-{\pr_{12}} \ar[d]_{\pr_{13}} & 
\bbG_m \xt \bbA^1 \ar[d]^{\id \xt h} \\ \bbG_m \xt V^\circ \ar[r]^{\pr_1} & \bbG_m } 
\]
so proper base change gives $\pr_{13!}\pr_{12}^* \cong \pr_1^* (\id\xt h)_!$.
We have an exact triangle 
\[ R \to p_2^*B \to (\id \xt 1)_! R(-1)[-1]\]
where $1 : \spec k \into \bbA^1$ is the complement of $u$.
Since $(\id \xt h)_! (m^*B) = 0$ by a change of variables, we deduce that 
\[ (\id\xt h)_!(m^*B \ot p_2^*B) \cong
(\id \xt h)_! (m^*B \ot (\id \xt 1)_! R)(-1)[-1] \cong 
j^*B(-1)[-1]. \]
Now it follows that $\rho^* J^*K \cong \pr_1^* j^*B(-d)[1]$.
\end{proof}

\begin{proof}[Proof of Theorem \ref{keythm}] 
The case $d=0$ is obvious since $h_!B =0$ and $0^*B \cong R$.
From now on we will assume that $d>0$. 
We will show that both sides of \eqref{eq:claim} are in the
essential image of the functor $\tau_{\le 0}J_*J^*$, i.e., there are isomorphisms 
\[ K \cong \tau_{\le 0}J_*J^*(K) \text{ and }
\rho_! \pr_1^*j^*B(-d)[1]\cong \tau_{\le 0}J_*J^*(\rho_! \pr_1^*j^*B(-d)[1]).\] 
The claimed existence of an isomorphism \eqref{eq:claim} will then follow from Lemma \ref{lem2}.

A stalk computation shows
that $\rho_!\pr_1^*j^*B(-d)[1]$ lives in non-positive degrees. 
We claim that the natural morphism
\begin{equation}\label{eq2}
 \rho_! \pr_1^*j^*B(-d)[1] \to \tau_{\le 0}(J_*J^*\rho_! \pr_1^*j^*B(-d)[1] ) 
\end{equation}
is an isomorphism. 
Let $0 : S \into V \xt_S V$ denote the zero section. 
From the exact triangle $0_!0^! \to \id \to J_*J^*$, it suffices
to show that $0^! \rho_!\pr_1^*j^* B \in D^{>2}_c(S)$. 
Observe that $\rho$ is $\bbG_m$-equivariant with respect to the $\bbG_m$-action on the
second coordinate of $\bbG_m \xt V$ and the diagonal action
of $\bbG_m$ on $V \xt_S V$. This implies that $\rho_! \pr_1^*j^*B$ is monodromic. 
Thus \[ 0^! \rho_! \pr_1^*j^* B \cong h_!j_! j^*B (-d)[-2d] \cong R(-d)[-2d-1]. \] 
Therefore $0^! \rho_! \pr_1^*j^*B \in D^{>2d}_c(S)$. 

One easily sees that $K_{(0,0)} \cong R(-d)$. 
Thus $K$ lives in non-positive cohomological degrees. 
To show that the natural morphism $K \to \tau_{\le 0}J_*J^*K$ is an isomorphism,
it suffices by the same argument as above to prove $0^! K \in D^{\ge 2d}_c(S)$.
One observes from the definition of $K$ that $K$ is monodromic with respect to
the diagonal $\bbG_m$-action on $V \xt_S V$.
Therefore \[ 0^!K \cong \wtilde\pi_!( \pr_{12}^*\mu^*B
\ot \pr_{23}^*\mu^*B)[2d] \]
where $\wtilde\pi : V \xt_S V^\vee \xt_S V \to S$ is the structure map.
By projection formula and proper base change, the right hand side is isomorphic to 
\[ \pi'_!(\mu^*B \ot \pr^{\vee,*}\pr^\vee_! \mu^*B)[2d]\] 
for $\pi':V \xt_S V^\vee \to S$ the structure map.
The fact that $h_!B=0$ implies that
$\pr^\vee_! \mu^*B$ is supported at $0(S) \subset V^\vee$, and
$0^*\pr^\vee_! \mu^*B \cong R(-d)[-2d]$. 
We deduce that
\[ 0^!K \cong \pi_!R(-d) \cong R(-2d)[-2d], \]
which proves the claim, and hence the theorem.
\end{proof}

\section{Properties of $\Four_B$} \label{sect:main}

\begin{rem}\label{rem:FBneq}
	The functor $\Four_{V/S,B}$ is not an equivalence on 
	$D^b_c(V) \to D^b_c(V^\vee)$. Consider the one-dimensional case $V = \bbA^1_S$. 
	Then $\Four_{V/S,B}(0_!R) = R[1]$ and $\Four_{V/S,B}(1_!R) = B[1]$. 
	We have $\Hom(R,B) \ne 0$ but $\Hom(0_!R, 1_!R) = 0$, so $\Four_{V/S,B}$ is not fully faithful.
\end{rem}

\subsection{Relation to quotient stacks} Let $p:V \to \eV = [V/\bbG_m]$ 
and $p^\vee : V^\vee \to \eV^\vee = [V^\vee/\bbG_m]$
denote the canonical projections to the quotient stacks. 
By \cite[Lemme 3.2]{Laumon}, Laumon's homogeneous transform 
$\Four_{\eV/S}:D^b_c(\eV) \to D^b_c(\eV^\vee)$ 
is canonically isomorphic to the functor 
\begin{equation}\label{eq:laumon}
	K \mapsto \pr^\vee_!(\pr^*K \ot \mu^*f_! B_S)[d] 
\end{equation}
where $f:\bbA^1_S \to \eA_S$ is the quotient morphism and $B_S$ denotes the base change
of $B$ from $\bbA^1_k$ to $\bbA^1_S$. We abuse notation and 
use $\pr^\vee : \eV^\vee \xt_S \eV \to \eV^\vee,\, \pr : \eV^\vee \xt_S \eV \to \eV$,
and $\mu : \eV^\vee \xt_S \eV \to \eA_S$ to also denote the induced maps on stacks.

\begin{prop} The composed functors
\[ (p^\vee)^* \circ \Four_{\eV/S} \text{ and } \Four_{V/S,B} \circ p^* 
	: D^b_c(\eV) \to D^b_c(V^\vee) \]
are canonically isomorphic.
\end{prop}
\begin{proof}
The proposition follows from \eqref{eq:laumon} by 
applying proper base change to the Cartesian squares
\[ \xymatrix{ [ V^\vee \xt_S V / \bbG_m] \ar[r] \ar[d] & \bbA^1 \ar[d]^f \\ \eV^\vee \xt_S \eV
\ar[r]^\mu & \eA_S } \quad\quad 
\xymatrix{ V^\vee \xt_S V \ar[r] \ar[d] & [V^\vee \xt_S V / \bbG_m] \ar[d] \\ 
V^\vee \ar[r] & \eV^\vee } \]
where $\bbG_m$ acts on $V^\vee \xt_S V$ anti-diagonally.
\end{proof}

\begin{prop} Let $V' = V \xt \bbA^1$ and let $\bbG_m$ act on both $V$ and $\bbA^1$.
We have a canonical open embedding $\nu : V \into [V'/\bbG_m] : v \mapsto (v,1)$.    
Similarly, we have $\nu^\vee : V^\vee \into [(V')^\vee/\bbG_m]$ defined by
$\nu^\vee(\xi) = (\xi,-1)$. The composed functor
\[ \xymatrix{ D^b_c(V) \ar[r]^-{\nu_!} & D^b_c([V'/\bbG_m]) \ar[rr]^-{\Four_{[V'/\bbG_m]/S}} &&
D^b_c([(V')^\vee/\bbG_m]) \ar[r]^-{(\nu^\vee)^*} & D^b_c(V^\vee) } \]
is isomorphic to $\Four_{V/S,B}$.
\end{prop}
\begin{proof}
Observe that $\nu$ factors into the composition of an open affine chart $V \into \bbP(V')$
and the open embedding $\bbP(V') = [(V'-0(S))/\bbG_m] \into [V'/\bbG_m]$. 
Similarly, we have a factorization of $\nu^\vee$. 
The proposition now follows from \cite[Proposition 1.6]{Laumon}, 
since the restriction of the incidence hyperplane in $\bbP((V')^\vee) \xt_S \bbP(V')$
to $V^\vee \xt_S V$ is $\mu^{-1}(\{1\})$.
\end{proof}

\subsection{An equivalence induced by $\Four_{V/S,B}$} Let $p : V \to \eV$ be as
in the previous subsection. 
\begin{prop} 
Let $\eC_V$ denote the full subcategory of $D^b_c(V)$ consisting of complexes $M$ such 
that $p_! M = 0$. The functor $\Four_{V/S,B}$ induces an equivalence
$\eC_V \to \eC_{V^\vee}$.
\end{prop}

\begin{proof}
Proper base change and projection formula imply that $\Four_{V/S,B}$ sends 
$\eC_V$ to $\eC_{V^\vee}$ and vice versa.
We also see by proper base change that $p^*p_!M \cong R * M$ for $M \in D^b_c(V)$, 
where $R$ is the constant sheaf on $\bbG_m$. 
From the exact triangle $1_!R(-1)[-2] \to R \to B$ we deduce that 
$j^*B * M \cong M(-1)[-1]$ for $M \in \eC_V$. 
Therefore Theorem \ref{keythm} implies that 
\[ \Four_{V^\vee/S,B} \circ \Four_{V/S,B}(M) \cong M(-d-1) \] for $M \in \eC_V$,
and we deduce the proposition.
\end{proof}

\subsection{Monodromic complexes}
We will show that $\Four_{V/S,B}$ also induces an equivalence on the subcategories of monodromic complexes. We use the notation and results of Appendix \ref{append}. 

\begin{thm} \label{thm:mon} 
(i) The functor $\Four_{V/S,B}$ preserves monodromicity, and the restriction
defines an equivalence $D^b_\mon(V) \to D^b_\mon(V^\vee)$. 
\smallskip

\noindent (ii) 
For $N \in D^b_\mon(V^\vee)$, the pro-object 
\begin{equation}\label{eq:B} \pr_!(\pr^{\vee,*}N \ot \mu^* j_*I^0)(d+1)[d+1] 
\end{equation}
is essentially constant\footnote{A pro-object is essentially constant if it
is isomorphic to an object of $D^b_\mon(V)$, which is considered as a pro-object via the
constant embedding.}. 
\smallskip

\noindent(iii) 
The functor $D^b_\mon(V^\vee) \to D^b_\mon(V)$ defined by \eqref{eq:B} is quasi-inverse 
	to $\Four_{V/S,B}$.
\end{thm}

Since $B$ is not monodromic, our first step is to compute the ``monodromization''
of $B$. 
\begin{lem} \label{lem:mon-B}
There is an isomorphism of pro-objects
\[ I^0 * B \cong j_*I^1(-1)[-1]. \] 
\end{lem}
\begin{proof} First we show that the restriction $I^0 * j^*B$ is isomorphic to 
$I^1(-1)[-1]$. The exact triangle $1_! R(-1)[-2] \to R \to B$
induces by convolution exact triangles 
\[ I^0_n(-1)[-2] \to I^0_n * R \to I^0_n * j^*B \] 
for $p\nmid n$.
Taking $\prolim$ and using Lemma \ref{lem:mon}, the first arrow
is isomorphic to the augmentation map $I^0(-1)[-2] \to R(-1)[-2]$. Therefore
we deduce that the pro-object $I^0 * j^*B$ is 
isomorphic to $I^1(-1)[-1]$.

To complete the proof, it suffices to show that the canonical morphism 
\[ I^0 * B \to j_*j^*(I^0 * B) \] is an isomorphism. This is
equivalent to proving that $0^! (I^0 * B) =0$.
Since $I^0 * B$ is monodromic, $0^!(I^0 * B) \cong h_!(I^0 * B)$. 
By the Kunneth formula, $h_!(I^0*B) \cong h_!j_! I^0 \ot h_! B= 0$.
\end{proof}

\begin{proof}[Proof of Theorem \ref{thm:mon}]
One easily sees that $\Four_{V/S,B}$ preserves monodromicity.
Theorem \ref{keythm} and Lemma \ref{lem:mon} together imply that for $M \in D^b_\mon(V)$, we have
\[ \Four_{V^\vee/S,B} \circ \Four_{V/S,B}(M) \cong I^1 * M(-d)[2]. \]
Since $I^{-1} * I^1 \cong I^0(-1)[-2]$ by Corollary \ref{cor:Ii+j}, we deduce that $\Four_{V/S,B}$ is an 
equivalence, with inverse functor $I^{-1} * \Four_{V^\vee/S,B} (d+2)[2]$.
Lemmas \ref{lem:mon-B} and \ref{lem:mon} imply that for $N \in D^b_\mon(V^\vee)$, we have
isomorphisms
\[ I^{-1}*\Four_{V^\vee/S,B}(N) \cong I^{-1}*\pr_!( \pr^{\vee,*}N \ot \mu^*j_* I^1) [d+1]. \]
Applying Corollary \ref{cor:Ii+j} again, we get (iii).
\end{proof}

\begin{rem}
	Observe that the formula \eqref{eq:B} is very similar to Beilinson's suggested
	definition of the monodromic Fourier transform in \cite{Beilinson2}.
\end{rem}

\begin{prop} \label{prop:B}
	The object $j^*B \in D_{ctf}(\bbG_m)$ satisfies the following properties:
	\begin{enumerate} 
		\item $j^*B$ is not invertible in the monoidal category $D_{ctf}(\bbG_m)$.
		\item $j^*B$ is invertible in the quotient of $D_{ctf}(\bbG_m)$ by the
			ideal generated by the constant sheaf $R$. 
		\item There are canonical isomorphisms $I^0_n * j^*B \cong I^1_n(-1)[-2]$
			for $p\nmid n$.
	\end{enumerate}
\end{prop}
\begin{proof}

We showed in Remark \ref{rem:FBneq} that $\Four_{\bbA^1,B}$ is not an 
equivalence on $D^b_c(\bbA^1)$. Since $\Four_{\bbA^1,B}^2(M)$ is isomorphic to 
$j^*B * M(-1)[1]$, we deduce
that $j^*B$ is not invertible in the monoidal category $D_{ctf}(\bbG_m)$.

From the exact triangle $1_! R(-1)[-2] \to R \to j^*B$ on $\bbG_m$, we see that 
in the quotient of $D_{ctf}(\bbG_m)$ by the ideal generated by $R$, the object $j^*B$
is isomorphic to $1_! R(-1)[-1]$, which is invertible. 

Lemma \ref{lem:mon-B} gives an isomorphism 
$I^0 * j^*B \cong I^1(-1)[-2]$. Convolving with $I^0_n$, we
get an isomorphism $I^0_n * j^*B \cong I^0_n*I^1$.
One observes that $I^0_n*I^1 \cong I^1_n(-1)[-2]$ by 
Corollary \ref{cor:Ii+j}.
\end{proof}

\section{Relation to Fourier-Deligne transform} 
\label{sect:FD}

Suppose that $k$ has characteristic $p>0$. 
Assume that $R$ contains a primitive $p$-th root of unity $\zeta$ (where ``primitive''
means that $\zeta-1$ is invertible). Let 
$\psi:\bbF_p \to R^\times$ be the corresponding additive character with $\psi(1)=\zeta$, 
and let $\eL_\psi$ denote the Artin-Schreier sheaf.
The usual Fourier-Deligne transform $\Four_{V/S,\eL_\psi} 
:D^b_c(V) \to D^b_c(V^\vee)$ is defined by 
\[ \Four_{V/S,\eL_\psi}(M) = \pr^\vee_!(\pr^* M \ot \mu^*\eL_\psi)[d]. \]

\begin{lem}\label{lem:P^2=B} 
There is a canonical isomorphism 
\[ \iota^*j^*\eL_\psi * \eL_\psi \cong B[-1] \] 
where $\iota : \bbG_m \to \bbG_m$ sends $\lambda \mapsto -\lambda^{-1}$.
\end{lem}
\begin{proof}
By a change of variables, $\iota^*j^*\eL_\psi * \eL_\psi$ is isomorphic to
$\Four_{\bbA^1,\eL_{\psi^{-1}}}(j_!j^*\eL_\psi)[-1]$. We have an exact triangle
\[ j_!j^*\eL_\psi \to \eL_\psi \cong \Four_{\bbA^1,\eL_\psi}(1_!R[-1]) \to 0_* R. \]
Applying $\Four_{\bbA^1,\eL_{\psi^{-1}}}$ and using the Fourier-Deligne inversion formula  
on the middle term, we have an exact triangle 
\[ \Four_{\bbA^1,\eL_{\psi^{-1}}}(j_!j^*\eL_\psi) \to 1_!R(-1)[-1] \to R[1]. \] 
This induces an isomorphism $\Four_{\bbA^1,\eL_{\psi^{-1}}}(j_!j^*\eL_\psi) \to u_*R = B$. 
Since $\Hom(1_! R(-1)[-1], R) = 0$, this isomorphism is unique. 
\end{proof}

\begin{cor} \label{cor:BL^2}
In characteristic $p>0$, we have a canonical isomorphism
\[ \Four_{V/S,B}(M) \cong \iota^* j^*\eL_\psi * \Four_{V/S,\eL_\psi}(M)[1]. \]
\end{cor}

\subsection{Monodromization of $\eL_\psi$ over $\bar k$} 
Suppose that $k$ is algebraically closed, so $A^0$ is simply a ring instead of
a sheaf of rings (i.e., there is no Galois action). 

\begin{lem} \label{lem:mon-AS}
There exists a (non-canonical) isomorphism of pro-objects
\[ I^0 * \eL_\psi \cong j_*I^0 [-1]. \] 
\end{lem}
\begin{proof}
As in the proof of Lemma \ref{lem:mon-B}, it suffices to prove the isomorphism 
after restriction to $\bbG_m$. Let $n$ be coprime to $p$. By proper base change,  
\[ 1^*(I^0_n * j^*\eL_\psi) \cong \Gamma_c(\bbG_m, I^0_n \ot_R j^*\eL_\psi) \]
where we observe that the pullback of $I^0_n$ under the
multiplicative inverse map $\bbG_m \to \bbG_m$ is isomorphic to $I^0_n$.
Since $I^0_n$ is tamely ramified at $\infty \in \bbP^1(k)$, the canonical
map \[ \Gamma_c(\bbA^1, j_!I^0_n \ot \eL_\psi) \to \Gamma(\bbA^1, j_!I^0_n \ot \eL_\psi) \]
is an isomorphism (cf.~proof of \cite[Lemma 7.1(1)]{KW}). 
In particular $\Gamma_c(\bbG_m, I^0_n \ot j^*\eL_\psi)$ lives in cohomological degrees
$0$ and $1$. Since $I^0_n \ot j^*\eL_\psi$ is locally constant and $\bbG_m$ is not complete,
$H^0_c(\bbG_m, I^0_n \ot j^*\eL_\psi)=0$. Thus $\Gamma_c(\bbG_m, I^0_n \ot j^*\eL_\psi)$
lives only in cohomological degree $1$. 
\smallskip

We now consider $I^0_n$ as a locally free sheaf of $A^0_n$-modules of rank $1$. 
If we let $\psi'$ denote the composition $\bbF_p \to R^\times \to (A^0_n)^\times$, 
then $\eL_\psi \ot_R A^0_n \cong \eL_{\psi'}$, where the latter is
the Artin-Schreier sheaf with respect to $\psi'$ as a locally free
sheaf of $A^0_n$-modules of rank $1$. Hence
$\eF:=I^0_n \ot_{A^0_n} j^*\eL_{\psi'}$, which is isomorphic to $I^0_n \ot_R j^*\eL_\psi$,
is a locally free sheaf of $A^0_n$-modules of rank $1$. 
In particular, $\eF \in D_{ctf}(\bbG_m, A^0_n)$ and 
$\Gamma_c(\bbG_m, \eF)[1]$ is quasi-isomorphic to a finite projective $A^0_n$ module $P$. 
Applying the Grothendieck-Ogg-Shafarevich formula 
\cite[Expos\'e X, Corollaire 7.2]{SGA-5}, one checks that the fiber of $P$ over any
point of $\spec A^0_n$ has dimension $1$. 
So there exists an isomorphism $P \cong A^0_n$ of $A^0_n$-modules.
Observe from the Cartesian square
\[ \xymatrix{ \bbG_m \xt \bbG_m \xt \bbG_m \ar@<2pt>[rr]^-{\theta(n) \xt \id_{\bbG_m}} 
\ar@<-2pt>[rr]_-{\pr_2 \xt \id_{\bbG_m}}  \ar[d]_{\id_{\bbG_m \xt m}} &&
\bbG_m \xt \bbG_m \ar[d]^m \\ \bbG_m \xt \bbG_m \ar@<2pt>[rr]^-{\theta(n)} 
\ar@<-2pt>[rr]_-{\pr_2} && \bbG_m } \]
that $I^0_n * j^*\eL_\psi$ is monodromic, and the monodromy action is induced by the monodromy
action on $I^0_n$. Hence by Corollary~\ref{cor:Locf}, 
there exists an isomorphism $I^0_n * j^*\eL_\psi[1] \cong I^0_n$.  

Suppose $n'$ is a multiple of $n$ and $p \nmid n'$. The kernel $\eK$ of the
surjection $I^0_{n'} \onto I^0_n$ is tamely ramified, so $H^2_c(\bbG_m, \eK \ot j^*\eL_\psi)=0$
by the same argument as above. We deduce that 
\[ I^0_{n'} * j^*\eL_\psi[1] \to I^0_n * j^*\eL_\psi[1] \] 
is a surjection of sheaves. Since $(A^0_{n'})^\times \to (A^0_n)^\times$ is also
surjective, we can find a projective system of isomorphisms 
$I^0_n * j^*\eL_\psi[1] \cong I^0_n$ 
inducing an isomorphism of pro-sheaves.
\end{proof}

\begin{cor} When $k$ is algebraically closed, there exists a (non-canonical) isomorphism 
between the functors
$\Four_{V/S,B}$ and $\Four_{V/S,\eL_\psi}$ restricted to $D^b_\mon(V) \to D^b_\mon(V^\vee)$. 
\end{cor}
\begin{proof}
Lemma \ref{lem:mon-B} and Remark \ref{rem:Itorsor} imply that there exists
an isomorphism $I^0 * B \cong j_*I^0[-1]$. The latter is also isomorphic to 
$I^0 * \eL_\psi$ by Lemma \ref{lem:mon-AS}. One easily sees that the Fourier-Deligne transform
preserves monodromicity, and the isomorphism of restricted functors follows
from Lemma \ref{lem:mon}.
\end{proof}

\subsection{The universal Gauss sum} Let $k$ once again be arbitrary.
Define the pro-object \[ \eG= I^0 * j^*\eL_\psi(1)[1]. \] 
Lemma \ref{lem:mon-AS} implies that
$\eG$ is a monodromic pro-sheaf, and there exists a trivialization 
$\eG \cong I^0$ after base changing from $k$ to $\bar k$. 
Under the equivalence of abelian categories in Corollary~\ref{cor:Locf}, we see 
that $\eG$ corresponds
to an invertible (locally free of rank $1$) $A^0$-module on $\spec k$. 
We are motivated by \cite[Expos\'e VI, \S 4]{SGA-4h} to think of 
$\eG$ as a ``universal Gauss sum''. 

Let $\iota: \bbG_m \to \bbG_m$ denote the multiplicative inverse map. 
Then Lemmas \ref{lem:mon-B} and \ref{lem:P^2=B} give a canonical isomorphism 
\[ \iota^* \eG * \eG \cong I^1[-2].\]
We also see that the Fourier-Deligne transform on monodromic complexes is 
isomorphic to the functor 
$M \mapsto \pr^\vee_!(\pr^*M \ot \mu^*j_*\eG)[d+1]$
on $D^b_\mon(V) \to D^b_\mon(V^\vee)$. 
By Corollary~\ref{cor:BL^2}, we have 
\[ \Four_{V/S,B}(M) \cong \iota^* \eG * \Four_{V/S,\eL_\psi}(M)[2]. \]
for $M$ monodromic.

\section{Relation to Fourier transform on $\eD$-modules}
\label{sect:Dmod}

Let $k$ be algebraically closed of characteristic $0$. 
We use $\eM(V)$ to denote the abelian category of quasicoherent right $\eD$-modules on $V$.
Let $\eL = \eD_{\bbA^1}/(1-\partial_x)\eD_{\bbA^1}$ be the exponential $\eD$-module
on $\bbA^1 = \spec k[x]$. The Fourier transform is the functor
$D\eM(V) \to D\eM(V^\vee)$ defined by 
\[ \Four_{V/S,\eL}(M) = \pr^\vee_*( \pr^! M \ot^! \mu^! \eL )[1-d]. \]
It is well known \cite[Lemme 7.1.4]{Katz-Laumon} that this functor can also be described using the
isomorphism between the algebras of polynomial differential operators 
$\eD_{V^\vee} \to \eD_V$ defined in local coordinates by 
\[ k[\xi_1,\dotsc,\xi_d,\partial_{\xi_1},\dotsc,\partial_{\xi_d}]\to
k[v_1,\dotsc,v_d,\partial_{v_1},\dotsc,\partial_{v_d}] : 
\xi_i \mapsto -\partial_{v_i},\, \partial_{\xi_i} \mapsto v_i. \]

In the $\eD$-module situation, the analog of $B$ is 
$u_!u^! (\omega_{\bbA^1})$, where $\omega_{\bbA^1}$ is the sheaf of
differentials on $\bbA^1$ viewed as a right $\eD$-module. 
We will also call this $\eD$-module $B$.
A simple calculation shows that\footnote{Beilinson observed that $B$ 
essentially describes the differential equation for a shift of the Heaviside step function.}
\[ B = k[x, \partial_x]/ \partial_x (x-1) k[x,\partial_x]. \]
We define $\Four_{V/S,B} : D\eM(V) \to D\eM(V^\vee)$ by 
\[ \Four_{V/S,B}(M) = \pr^\vee_*(\pr^!M\ot^! \mu^!B)[1-d]. \]
Consider $D\eM(\bbG_m)$ with the monoidal structure induced by convolution without
compact support 
$L * K := m_*(L \bt K)$.
This monoidal category acts on $D\eM(V)$ by $L*M = \theta(1)_*(L \bt M)$.
The proof of Lemma~\ref{lem:P^2=B} can be easily modified to prove the following analog
of the lemma and Corollary~\ref{cor:BL^2}.

\begin{prop} There is a canonical isomorphism 
	\[ \iota^* j^*\eL * \eL \cong B \]
where $\iota : \bbG_m \to \bbG_m$ sends $\lambda \mapsto -\lambda^{-1}$. 
Consequently, we have a canonical isomorphism
\[ \Four_{V/S,B}(M) \cong \iota^*j^*\eL * \Four_{V/S,\eL}(M). \]
\end{prop}

\subsection{Mellin transform of $j^*B$}
Let $\fB$ denote the Mellin transform of $j^*B$, viewed
as a $\bbZ$-equivariant quasicoherent $\eO$-module on $\bbA^1 = \spec k[s]$.
The Mellin transform functor 
\[ \fM : \eM(\bbG_m) \to \QCoh(\bbA^1)^\bbZ \]
is defined by considering $\eD(\bbG_m)$ as the algebra of difference operators 
$\eD= k[s]\brac{T,T^{-1}}/(sT - T(s+1))$ under the identifications $s = x\partial_x$
and $T = x$. 
We consider the derived category of $\bbZ$-equivariant $\eO_{\bbA^1}$-modules 
$D(\QCoh(\bbA^1)^\bbZ)$ 
with monoidal structure induced by the usual derived tensor product over $k[s]$.
This monoidal structure corresponds to the convolution product 
on $D\eM(\bbG_m)$. More precisely,
$\fM(L*K) \cong \fM(L) \ot_{k[s]} \fM(K)$.

We start by proving the following
proposition, which is an analog of Proposition \ref{prop:B} in the $\eD$-module setting.

\begin{prop} \label{prop:Dmod} The module $\fB$ satisfies the following properties: 
	\begin{enumerate}
		\item $\fB$ is not invertible in $D(\QCoh(\bbA^1)^\bbZ)$. 
		\item The restriction of $\fB$ to $\bbA^1 - \bbZ := \spec k[s][s^{-1},(s\pm 1)^{-1},
			\dotsc]$ is invertible. 
		\item For any $\chi \in k$ and $n \in \bbN$, there exists an isomorphism 
			\[ \bigoplus_{i\in \bbZ} k[s]/(s-\chi-i)^n \cong
			\bigoplus_{i \in \bbZ} \fB \ot_{k[s]} k[s]/(s-\chi-i)^n \] of $\eD$-modules,
			where $T$ acts on $k[s]$ by translation. 
	\end{enumerate}
\end{prop}

In order to prove the proposition, we will need an explicit description of $\fB$. 
Consider $k(s)$ as a right $\eD$-module where $T$ acts by translation.
Let $\fB'$ denote the $\eD$-submodule of $k(s)$ generated by $\frac 1 s$, or
equivalently, the $k[s]$-submodule generated by $\frac 1 {s+i}$ for all $i \in \bbZ$.

\begin{lem} \label{lem:B'}
There exists an isomorphism of $\eD$-modules $\fB \cong \fB'$. 
\end{lem}
\begin{proof}
We have $\partial_x x = x\partial_x + 1$ so 
$\partial_x (x-1) = (s+1)-T^{-1} s$ in $\eD$. Therefore 
\[ \fB = \eD/((s+1)-T^{-1}s)\eD. \]
Let $\mathbf{1}$ denote the generator of $\fB$.
Conjugating $sT=T(s+1)$ in $\eD$ by $T^{-1}$ gives $T^{-1}s = (s+1)T^{-1}$ in $\eD$. 
Using this equality, $\mathbf{1}(s+1) = \mathbf 1 T^{-1}s = \mathbf 1 (s+1)T^{-1}$ in $\fB$,
and acting on the right by $T$ gives $\mathbf 1 (s+1)T = \mathbf 1(s+1)$. 
Using these relations, we deduce that $\fB$ is generated over $k$ by $\mathbf 1 T^i$ for $i \in \bbZ$
and $\mathbf 1 s^j$ for $j> 0$. 
 Then 
$\mathbf 1 \mapsto \frac 1 {s+1}$ defines a morphism of $\eD$-modules $\fB \to k(s)$.
Since $\frac 1 {s+i}$ for $i \in \bbZ$ and $s^j$ for $j\ge 0$ are $k$-linearly independent
in $k(s)$, we see that this morphism is an injection $\fB \into k(s)$. 
The image is $\fB'$.
\end{proof}

\begin{proof}[Proof of Proposition \ref{prop:Dmod}] Suppose that $\fB$
is invertible in $D(\QCoh(\bbA^1)^\bbZ)$, i.e., there exists an object $N$ of
this monoidal category such that $\fB \ot_{k[s]} N \cong k[s]$. 
Then $N \cong \Hom_{k[s]}(k[s],N) \cong \Hom_{k[s]}(\fB,k[s])$.
There are no nonzero morphisms from $\fB'$ to $k[s]$, so $H^0 N=0$.
On the other hand, since $k(s) \ot_{k[s]} \fB' \cong k(s)$, we 
have $k(s) \ot_{k[s]} N \cong k(s)$, which implies that $H^0 N \ne 0$. We thus get a contradiction,
so $\fB$ is not invertible. 

Since $\eO(\bbA^1-\bbZ) = k[s][s^{-1},(s\pm 1)^{-1},\dotsc] \subset k(s)$, 
we see that \[ \eO(\bbA^1-\bbZ) \ot_{k[s]} \fB' = \eO(\bbA^1-\bbZ) \subset k(s) \] is the identity 
object, proving (2). 
 
The direct sums in (3) only depend on the class $\wbar \chi$ of $\chi$ in $k/\bbZ$.
If $\wbar\chi = 0+\bbZ$ we will assume that $\chi=0$. 
Let $\fB_i \subset \fB'$ denote the $k[s]$-submodule generated by $\frac 1 {s-i}$. 
Then $\fB'/\fB_i$ is isomorphic to the direct sum of skyscraper modules $k[s]/(s-j)$ for 
integers $j\ne i$. Thus $(\fB'/\fB_i) \ot_{k[s]} k[s]/(s-\chi-i)^n = 0$. 
On the other hand $\fB_i$ is free, so $\fB' \ot_{k[s]} k[s]/(s-\chi-i)^n$ is free 
with generator $\frac 1 {s-i} \ot 1$. These basis elements give
our desired isomorphism, which evidently commutes with the action of $T$.
\end{proof}

\subsection{Monodromization} The $\bbG_m$-action on $V$ induces
an algebra map $k[s] \to \eD_V$, where  
$s = x\partial_x$ is the invariant vector field on $\bbG_m$.
We say that $M \in \eM(V)$ is monodromic if every local section $m \in M$
is killed by some nonzero polynomial in $s = x \partial_x$. 
In other words, $M$ is monodromic if it is a torsion module over $k[s]$. 
This definition of monodromic was introduced by Verdier \cite{V130}.
Define an object of $D\eM(V)$ to be monodromic if each of its cohomology $\eD$-modules
is monodromic. We denote this full subcategory by $D_\mon \eM(V)\subset D\eM(V)$. 
 
For any $\chi \in k$ and $n \in \bbN$, let 
$A_{\chi,n} \subset k(s)$ consist of those rational functions with poles
of order $\le n$ at $\chi + \bbZ$ and no other poles. 
Define $I^{0,n}_\chi \in \eM(\bbG_m)$ to be the inverse Mellin transform 
$\fM^{-1}(A_{\chi,n}/k[s])$. 
The inclusions $A_{\chi,n} \to A_{\chi,n+1}$ induce morphisms $I^{0,n}_\chi
\to I^{0,n+1}_\chi$, which form an inductive system of $\eD$-modules.
Define 
\[ I^0 = \bigoplus_{\wbar\chi \in k/\bbZ} \dlim_n I^{0,n}_\chi \in \eM(\bbG_m) \]
where $\chi\in k$ is any lift of $\wbar\chi$. It follows that
$\fM(I^0) = k(s)/k[s]$.

Let $\underline 1$ be the unit object in the monoidal category $D\eM(\bbG_m)$, so  
$\fM(\underline 1) = k[s]$. The canonical extension of
$k(s)/k[s]$ by $k[s]$ defines an extension of $I^0$ by $\underline 1$ and therefore a
morphism \[\vareps :I^0\to \underline 1[1].\]
The monoidal category $D\eM(\bbG_m)$ acts on $D\eM(V)$ by convolution (without
compact support).

\begin{lem} \label{lem:1mon}
An object $M \in D\eM(V)$ is monodromic if and only if the morphism 
$I^0 * M \to M[1]$ induced by $\vareps$ is an isomorphism.
\end{lem}
\begin{proof}
A calculation using the relative de Rham complex with respect to the action map $\bbG_m \xt V \to V$ 
shows that for any $M \in D\eM(V)$ and $N \in D\eM(\bbG_m)$, 
there is a canonical isomorphism $N * M \cong \fM(N) \ot_{k[s]} M$ in the derived category of (sheaves of)
$k[s]$-modules. 
This implies that the cocone of the morphism $I^0 * M \to M[1]$ is isomorphic (in the derived
category of $k[s]$-modules) to $k(s) \ot_{k[s]} M$. 
But $k(s)$ is flat over $k[s]$, so the vanishing of the cohomologies of $k(s)\ot_{k[s]} M$ is
equivalent to the cohomologies of $M$ being torsion modules over $k[s]$. 
\end{proof}

See \cite{Beilinson}, \cite{Lichtenstein}, and \cite[C.2]{DG} for 
further details in the unipotently monodromic case (when $\chi=1$). 

\begin{lem} \label{lem:Dmodmon}
There exists an inductive system of isomorphisms 
	\[ I^{0,n}_\chi * B \cong j_! I^{0,n}_\chi \cong I^{0,n}_\chi * \eL. \] 
\end{lem}
\begin{proof} 
Since $h_*B = h_*\eL = 0$, it suffices as in Lemma \ref{lem:mon-B} to 
give isomorphisms of the above objects after restriction to $\bbG_m$. 
In fact, it suffices to construct isomorphisms between the Mellin transforms of these
restrictions, i.e., isomorphisms $\fM(I^{0,n}_\chi * j^*B) \cong \fM(I^{0,n}_\chi)
\cong \fM(I^{0,n}_\chi * j^*\eL)$. This is equivalent to
constructing isomorphisms 
\begin{align} \label{eq:1}
	\fM(I^{0,n}_\chi) \ot_{k[s]} \fB &\cong \fM(I^{0,n}_\chi),\quad \fB := \fM(j^*B), \\
	\fM(I^{0,n}_\chi) \ot_{k[s]} E &\cong \fM(I^{0,n}_\chi),\quad E := \fM(j^*\eL). 
\label{eq:2} 
\end{align}
Note that we have isomorphisms
\begin{equation}\label{eq:3}
	\fM(I^{0,n}_\chi) = A_{\chi,n}/k[s] \cong \bigoplus_{i\in \bbZ} k[s]/(s-\chi-i)^n. 
\end{equation}
Combining \eqref{eq:3} and Proposition \ref{prop:Dmod}(3), one gets \eqref{eq:1}. 
Let us construct \eqref{eq:2}. We have 
\[ E = \eD/(1-T^{-1}s)\eD. \]
Let $\mathbf 1$ be the generator of $E$. 
Let $E_i \subset E$ denote the free $k[s]$-submodule generated by $\mathbf 1 T^{-i-1}$
for $i \in \bbZ$. If $\chi \in \bbZ$, set $\chi=0$.
From the relation $\mathbf 1 T^{-i} = \mathbf 1 T^{-i-1}(s-i)$, we deduce that
$E/E_i$ is supported away from $\chi+i$, so $(E/E_i) \ot_{k[s]} k[s]/(s-\chi-i)^n = 0$. 
Hence $E \ot_{k[s]} k[s]/(s-\chi-i)^n$ is freely generated by $\mathbf 1 T^{-i-1} \ot 1$, and 
this gives us \eqref{eq:2}.
\end{proof}

Lemma \ref{lem:Dmodmon} implies in particular that $I^0 * B \cong I^0 * \eL$. 
We deduce from Lemma \ref{lem:1mon} that $\Four_{V/S,B}$ agrees with $\Four_{V/S,\eL}$ 
on $D_\mon\eM(V)$.

\begin{cor} \label{cor:FB-FL}
There is an isomorphism 
\[ \Four_{V/S,B} \cong \Four_{V/S,\eL} \]
of functors $D_\mon\eM(V) \to D_\mon \eM(V^\vee)$. 
\end{cor}

\appendix

\section{The monodromic subcategory} \label{append}

In this appendix we prove the facts we need about (non-unipotently) monodromic complexes.
For a more complete account of the unipotently monodromic story, see \cite{BY,Beilinson}.

\subsection{Free monodromic objects}
Let $p$ be the characteristic of $k$, which may be $0$. 
For $p \nmid n$, let $A^0_n$ be the group algebra $R[\mu_n]$ considered as a sheaf on $\spec k$, 
i.e., a $\Gal(\bar k/k)$-module. Put \[ A^0 = \ilim_{p\nmid n} A^0_n. \] 
Consider $\bbT: =\ilim_{p\nmid n} \mu_n(\bar k)$ the tame fundamental group of $\bbG_{m,\bar k}$.
For any $\gamma \in \bbT$, let $\wtilde \gamma$ denote the corresponding invertible element in $A^0(\bar k)$. 
Pick a topological generator $t \in \bbT$. 
Note that $\wtilde t -1$ is not a zero divisor in $A^0$, so $A^0$ injects to the localization 
$A = (A^0)_{\wtilde t-1}$. Define \[ A^i = (\wtilde t-1)^i A^0 \subset A \] 
for $i \in \bbZ$ and set $A^i_n = A^i \ot_{A^0} A^0_n$ for $p \nmid n$. 
The definition of $A^i$ is independent of the choice of $t$, and $A^i$
is a $\Gal(\bar k/k)$-module. Note that
$A^1$ is the kernel of the quotient map $A^0 \to A^0_1 = R$.

\begin{rem} The ring $A^0(\bar k)$ is isomorphic to the product
	of the completions of $R[t,t^{-1}]$ at all maximal ideals $\fm$ such that
	$t^n \equiv 1 \bmod \fm$ for some $p \nmid n$. The maximal ideals $\fm$ correspond
	to the eigenvalues of the monodromy action.
\end{rem}

For $i \in \bbZ$ and $p \nmid n$, let $I^i_n$ be the local system on $\bbG_m$ such that the  
fiber at $1 \in \bbG_m(k)$ is $A^i_n$ and the monodromy action of $\gamma \in \bbT$ 
is multiplication by $\wtilde \gamma$. We define $I^i$ to be the pro-sheaf 
\[ \prolim_{p\nmid n} I^i_n. \] 

\begin{rem}\label{rem:Itorsor}
After base change from $\spec k$ to $\spec \bar k$, 
the local systems $I^0_n$ and $I^i_n$ are isomorphic via
multiplication by $(\wtilde t-1)^i$, and this induces an isomorphism $I^0 \cong I^i$. 
The isomorphism is not canonical as it depends on the choice of $t$.
\end{rem}

\begin{lem} \label{lem:mon}
There is a canonical isomorphism of pro-objects
\[ I^0 * M \cong M(-1)[-2] \] 
for $M \in D^b_\mon(V)$ considered as a constant pro-object. 
\end{lem}
\begin{proof}
Let $e_n : \bbG_m \to \bbG_m$ denote the $n^\mathrm{th}$ power map. 
Note that $e_{n!}R \cong I^0_n$ for $p\nmid n$. 
Since $M$ is monodromic, there exists $n_0$ coprime to $p$ such that 
$\theta(n_0)^*M \cong \pr_2^*M$. Then
\[ \prolim  (e_{n!}R) * M \cong \prolim \theta(n)_!\pr_2^*M 
\cong M(-1)[-2], \]
where we use the fact that the pro-object $\prolim \Gamma_c(\bbG_m,R)$
is essentially constant and isomorphic to $R(-1)[-2]$ (cf. \cite[Lemme 5.2]{Verdier}).
\end{proof}

\subsection{Monodromic sheaves as $A^0$-modules} 
\label{sect:Amod}
Let $\Mod_\tau(A^0)$ denote the abelian category of sheaves of discrete $A^0$-modules
on $\spec k$, where $A^0$ is equipped with the projective limit topology, and 
let $\Sh(\bbG_m)$ denote the abelian category of sheaves of $R$-modules
on $\bbG_m$. We have a canonical exact functor 
\[ \Loc:\Mod_\tau(A^0) \to \Sh(\bbG_m). \]
Define another functor $\fM : \Sh(\bbG_m) \to \Mod_\tau(A^0)$ 
by \[ \fM(\eF) = \dlim h'_* e_{n,*}e_n^* \eF \]
where $h':\bbG_m \to \spec k$ is the structure map and $A^0$ acts on $e_{n,*}e_n^*\eF$ 
by transport of structure. 
We deduce from \'etale descent that $\Loc$ is left adjoint to $\fM$. 
Passing to derived categories, the derived functors are still adjoint, and
we also denote them by 
\[ \Loc : D\Mod_\tau(A^0) \leftrightarrows D(\bbG_m) : \fM. \]
Note that $\fM : D(\bbG_m) \to D\Mod_\tau(A^0)$ is equal to the composition of the 
exact functor $\dlim e_{n,*}e_n^*$ with the derived functor $h'_*$.

\begin{prop} The derived functor $\Loc : D\Mod_\tau(A^0) \to D(\bbG_m)$
is fully faithful.  
\end{prop}
\begin{proof} We need to show that the unit of adjunction $L \to \fM \circ \Loc(L)$
is an isomorphism for $L \in D\Mod_\tau(A^0)$. 
We can assume that $k$ is algebraically closed.
Since $\Loc$ and $\fM$ both commute with filtered colimits, we may further suppose that 
$L$ is a finite module concentrated in degree $0$. 
Then there exists $n_0$ not divisible by $p$ such that the action of $A^0$ on 
$L$ factors through $R[\mu_{n_0}]$. If $n$ is a multiple of $n_0$ then $e_n^* \Loc(L) 
\cong \underline L$, where $\underline L$ is the constant sheaf on $\bbG_m$ with stalk $L$. 
The proposition now follows from the fact that for any finite abelian group $L$ of
order prime to $p$, one has $\dlim H^0\Gamma(\bbG_m,e_{n,*}e_n^* \underline L) \cong \underline L$ 
and $\dlim H^i\Gamma(\bbG_m,e_{n,*}e_n^*\underline L)=0$ for $i\ne 0$.
\end{proof}

\begin{cor} \label{cor:Locf}
The restriction of $\Loc$ induces an equivalence 
between the subcategory of $D^b\Mod_\tau(A^0)$ consisting of complexes whose 
cohomology sheaves have finite stalks and $D^b_\mon(\bbG_m)$. 
Taking hearts with respect to the standard $t$-structures of the above triangulated
categories, we get an isomorphism between the abelian category of 
sheaves of $A^0$-modules on $\spec k$ with finite stalk and the abelian category of 
monodromic sheaves on $\bbG_m$.
\end{cor}

The monoidal structure on $D\Mod_\tau(A^0)$ with respect to (derived) tensor product
over $A^0$ corresponds under $\Loc$ to convolution on $D(\bbG_m)$. 

\begin{lem} \label{lem:Aot}
For $L,K \in D\Mod_\tau(A^0)$ there exists a canonical isomorphism 
\[ \Loc(L) * \Loc(K) \cong \Loc(L \ot_{A^0} K ) (-1)[-2]. \]
\end{lem}
\begin{proof} Consider the functor $\Loc_{\bbG_m \xt \bbG_m}: D\Mod_\tau(A^0 \what\ot_R A^0)
\to D(\bbG_m \xt \bbG_m)$, which is defined similarly to the above functor 
$\Loc = \Loc_{\bbG_m}$. Applying $\Loc_{\bbG_m \xt \bbG_m}$ to 
the natural map $L \ot_R K \to L \ot_{A^0} K$, we get a map 
$\Loc(L) \bt \Loc(K) \to m^*\Loc(L \ot_{A^0} K)$ in $D(\bbG_m \xt \bbG_m)$.
Recall that since $m$ is smooth, $m^*\Loc(L \ot_{A^0}K) \cong m^!\Loc(L\ot_{A^0}K)(-1)[-2]$.
Therefore the $(m_!,m^!)$-adjunction induces a morphism
\[ \Loc(L) * \Loc(K) \to \Loc(L\ot_{A^0}K) (-1)[-2]. \]
To check this is an isomorphism, we can assume $k$ is algebraically closed 
and take $L=K=A^0_n$ for $p\nmid n$ since
the functors on both sides commute with filtered colimits and have finite cohomological amplitude.
Under these assumptions, the isomorphism is an easy computation.
\end{proof}

\begin{cor} \label{cor:Ii+j}
There is a canonical projective system of isomorphisms 
\[ I^i * I^j_n \cong I^{i+j}_n(-1)[-2] \]
for $p\nmid n$ and any integers $i$ and $j$. 
Consequently there is an isomorphism of pro-objects 
\[ I^i * I^j \cong I^{i+j}(-1)[-2]. \]
\end{cor}
\begin{proof} Fix $p \nmid n$. 
	By Lemma \ref{lem:Aot}, the first isomorphism is equivalent to an isomorphism
\[ \prolim_{p\nmid m} A^i_m \ot_{A^0} A^j_n \cong A^{i+j}_n \]
as pro-objects in $D\Mod_\tau(A^0)$. 
Remark \ref{rem:Itorsor} and Lemma \ref{lem:mon} imply that it suffices to 
consider the cohomology in degree $0$, i.e., we consider the non-derived tensor product
on the LHS. Then $H^0(A^i_m \ot_{A^0} A^j_n) \cong A^{i+j}_n$ for $n \mid m$ by
definition. These isomorphisms are evidently compatible with changes in $n$, so the
rest of the corollary follows.
\end{proof}

\bibliographystyle{alpha}
\bibliography{topic}

\end{document}